\documentclass{article}
\usepackage[utf8]{inputenc}
\usepackage[T1]{fontenc}
\usepackage{lmodern}
\usepackage[style=alphabetic,giveninits=true,url=false,isbn=false,doi=false]{biblatex}
\usepackage{url}
\usepackage{graphics}
\usepackage{xcolor}
\usepackage{mathtools,amssymb}
\usepackage{hyperref}
\usepackage{amsthm}

\newtheorem{thm}{Theorem}
\newtheorem*{thm*}{Theorem}

\newtheorem{lem}[thm]{Lemma}
\newtheorem{cor}[thm]{Corollary}

\newcommand{\Sp}{S} 		
\newcommand{\prob}{\mathbf{P}}	
\renewcommand{\P}{\prob}	
\newcommand{\Q}{\mathbf{Q}}	
\newcommand{\esp}{\mathbf{E}}	
\newcommand{\PS}{\mathcal{P}_2(\Sp)}			
\newcommand{\bs}{b^\star}	

\newcommand{\R}{\mathbf R}
\newcommand{\M}{\mathbf M}

\newcommand{\e}{\varepsilon}
\newcommand{\eps}{\varepsilon}

\newcommand{\lin}{\mathrm{Lin}}

\renewcommand{\phi}{\varphi}

\DeclareMathOperator{\supp}{supp}

\title{A note on flatness of non separable tangent cone at a barycenter}
\author{Thibaut Le Gouic \thanks{\url{thibaut.le_gouic@math.cnrs.fr}}
\thanks{National Research University Higher School of Economics (HSE), Faculty of Computer Science, Moscow, Russia} \thanks{Centrale Marseille, I2M, UMR 7373, CNRS, Aix-Marseille univ., Marseille, 13453, France}}

\date{\today}
\begin{document}
\maketitle

\abstract{Given a probability measure $\P$ on an Alexandrov space $\Sp$ with curvature bounded below, we prove that the support of the pushforward of $\P$ on the tangent cone $T_{\bs}\Sp$ at its (exponential) barycenter $\bs$ is a subset of a Hilbert space, without separability of the tangent cone.}

\begin{refsection}
\section{Introduction}
Barycenter of a probability measure $\P$ (a.k.a. Fréchet means) provides an extension of expectation on Euclidean space to arbitrary metric spaces.
We present here a useful tool for the study of barycenters on Alexandrov spaces with curvature bounded below: the support of $\log_{\bs}\#\P$ in the tangent cone at the barycenter is included in a Hilbert space.
This result has been stated in \cite{yokota_rigidity_2012} as Theorem 45, however the proof is not written.
Moreover, there is an extra assumption of support of $\log_{\bs}\#\P$ being separable, which does not even seem to be a consequence of the support of $\P$ being separable.
This paper present a proof of this result, without this extra separable assumption.
The proof is essentially the one of Theorem 45 of \cite{yokota_rigidity_2012}, with needed approximations dealt with a bit differently.

\section{Setting and main result}

We use a classical notion of curvature bounded below for geodesic spaces, referred to as Alexandrov curvature.
We recall several notions whose formal definitions can be found for instance in \cite{burago_course_2001} or in the work in progress \cite{alexander_alexandrov_2019}.

For a metric space $(\Sp,d)$, we denote by $\PS$ that set of probability measures $\P$ on $\Sp$ with finite moment of order $2$ (i.e. there exists $x\in\Sp$ such that $\int d^2(x,y)d\P(y)<\infty$).
The support of a measure $\P$ will be denoted by $\supp\P$.

A \emph{geodesic space} is a metric space $(\Sp,d)$ such that every two points $x,y\in\Sp$ at distance is connected by a curve of length $d(x,y)$.
Such shortest curves are called \emph{geodesics}.
For $\kappa\in\R$, the \emph{model space} $(\M_\kappa,d_\kappa)$ denotes the $2$-dimensional surface of constant Gauss curvature $\kappa$.
A geodesic space $(\Sp,d)$ is an \emph{Alexandrov space with curvature bounded below by} $\kappa\in \R$ if for every triangle ($3$-uple) $(x_0,x_1,y)\in\Sp$, and a constant speed geodesic $(x_t)_{t\in [0;1]}$ there exists an isometric triangle $(\tilde x_0,\tilde x_1,\tilde y)\in\M_\kappa$, such that the geodesic $(\tilde x_t)_{t\in[0;1]}$ satisfies for all $t\in [0;1]$,
\[
d(y,x_t)\ge d(\tilde y,\tilde x_t).
\]
For such spaces, angles between two unit-speed geodesics $\gamma_1,\gamma_2$ starting at the same point $p\in\Sp$ can be defined as follows:
\[
\cos\angle_p(\gamma_1,\gamma_2)=\lim_{t\rightarrow 0}\frac{d^2(\gamma_1(t),p)+d^2(\gamma_2(t),p)-d^2(\gamma_1(t),\gamma_2(t))}{2d(p,\gamma_1(t))d(p,\gamma_2(t))},
\]
where angle $\angle_p(\gamma_1,\gamma_2)\in[0;\pi]$.
Denote by $\Gamma_p$ the set of all unit-speed geodesics emanating from $x$.
Using angles, we can define the \emph{tangent cone $T_p\Sp$ at $p\in\Sp$} as follows.
First define $T'_p\Sp$ as the (quotient) set $\Gamma_x\times \R^+$, equipped with the (pseudo-)metric defined by
\[
\|(\gamma_1,t)-(\gamma_2,s)\|^2_p := s^2 + t^2 - 2s.t\cos\angle_p(\gamma_1,\gamma_2).
\]
Then, the \emph{tangent cone $T_p\Sp$} is defined as the completion of $T_p\Sp$.
We will use the notation for $u,v\in T_p\Sp$,
\[
\langle u,v\rangle_p:=\frac{1}{2}(\|u\|^2+\|v\|^2-\|u-v\|^2),
\]
We will often identify a point $\gamma(t)\in\Sp$ with $(\gamma,t)\in T_p\Sp$.
Although such $\gamma$ might not be unique, we will implicitly assume the choice of a measurable map $\log_p:\Sp\rightarrow T_p\Sp$, called \emph{logarithmic map}, such that for all $x\in\Sp$, there exists a geodesic $\gamma$ emanating from $p$ such that, for some $t>0$, $\gamma(t)=x$ and
\[
\log_p(x)=(\gamma,t).
\]
Then the \emph{pushforward} of $\P$ by $\log_p$ will be denoted by $\P\#\log_p$.

The tangent cone is not necessarily a geodesic space (see \cite{halbeisenTangent2000}), however, it is included in a geodesic space - namely the ultratangent space (see for instance Theorem 14.4.2 and 14.4.1 of \cite{alexander_alexandrov_2019}) that is an Alexandrov space with curvature bounded below by 0.

The tangent cone $T_p\Sp$ contains the subspace $\lin_p$ of all points with an \emph{opposite}, formally defined as follows.
A point $u$ belongs to $\lin_p\subset T_p\Sp$ if and only if there exists $v\in T_p\Sp$ such that $\|u\|_p=\|v\|_p$ and
\[
\langle u,v\rangle_p = -\|u\|_p^2.
\]
Our main result is based on the following Theorem.
\begin{thm*}[Theorem 14.5.4 in \cite{alexander_alexandrov_2019}]
The set $\lin_p$ equipped with the induced metric of $T_p\Sp$ is a Hilbert space.
\end{thm*}

A point $\bs\in\Sp$ is a \emph{barycenter} of the probability measure $\P$ if for all $b\in\Sp$
\[
\int d^2(x,\bs)d\P(x)\le \int d^2(x,b)d\P(x).
\]
Such barycenter might not be unique, neither exist.
However, when they exist, they satisfy
\begin{equation}\label{eq:expbar}
\int \langle x,y\rangle_{\bs} d\P\otimes\P(x,y)=0.
\end{equation}
A point $\bs\in\Sp$ satisfying \eqref{eq:expbar} is called an \emph{exponential barycenter} of $\P$.

We can now state our main result.
\begin{thm}
\label{thm:main}
Let $(\Sp,d)$ be an Alexandrov space with curvature bounded below by some $\kappa\in\R$ and $\P\in\PS$.
If $\bs\in\Sp$ is an exponential barycenter of $\P$, then the $\supp(\log_{\bs}\#\P)\subset \lin_{\bs}\Sp$.
In particular, $\supp(\log_{\bs}\#\P)$ is included in a Hilbert space.
\end{thm}

This result allows to prove the following Corollary, that has been implicitly used in \cite{ahidar-coutrixetalRate2018}.

\begin{cor}[Linearity]
\label{cor:lin}
Let $b\in T_{\bs}\Sp$
Then, the map $\langle .,b\rangle_{\bs}:\lin_{\bs}\rightarrow \R$ is continuous and linear.
In particular, if $\bs$ is an exponential barycenter of $\P$, then
\[
\int \langle x,b\rangle_{\bs}d\P(x)=0.
\]
\end{cor}

\section{Proofs}
Recall that we always identify a point in $\Sp$ and its image in the tangent cone $T_{p}\Sp$ by the $\log_p$ map.

\begin{proof}[Proof of Corollary \ref{cor:lin}]
We check that $x\mapsto \langle x,b\rangle_{\bs}$ is a convex and concave function in $\lin_{\bs}\Sp$.
Let $t\in (0,1)$, $x_0,x_1$ in $\lin_{\bs} \Sp$, and set $x_t=(1-t)x_0+tx_1$.
Since the tangent cone is included in an Alexandrov space with curvature bounded below by $0$ on one hand, and $\lin_{\bs}$ is a Hilbert space on the other hand,
\begin{align*}
\langle x_t,b\rangle_{\bs}&=\frac{1}{2}\left(\|x_t\|_{\bs}^2+\|b\|_{\bs}^2-\|x_t-b\|^2\right)\\
&\le \frac{1}{2} \left((1-t)(\|x_0\|_{\bs}^2-\|x_0-b\|_{\bs}^2)+t(\|x_1\|_{\bs}^2-\|x_1-b\|^2)+\|b\|_{\bs}^2\right)\\
&=(1-t)\langle x_0,b\rangle_{\bs}+t\langle x_1,b\rangle_{\bs}.
\end{align*}
The same lines applied to $-x_0$ and $-x_1$ gives the converse inequality
\[
\langle -x_t,b\rangle_{\bs}\le (1-t)\langle -x_0,b\rangle_{\bs} + t\langle -x_1,b\rangle_{\bs}.
\]
The second statement follows from the fact that $\bs$ is a Pettis integral of the pushforward of $\P$ onto $\lin_{\bs}\subset T_{\bs}\Sp$, as a direct consequence of Theorem \ref{thm:main}.
\end{proof}

\begin{proof}[Proof of Theorem \ref{thm:main}]
Let $x\in \supp \P$.
For $U=\{x\}$, use Lemma \ref{lem:approx} with $\Q=\P$ and $B_\delta$ a ball of radius $\delta$ around $x$ in $T_{\bs}\Sp$, to get a sequence $(y_\delta^n)_n\subset T_{\bs}\Sp$.
Then,
\begin{align*}
\limsup_n\cos \angle(\uparrow_{\bs}^x,&\uparrow_{\bs}^{y_\delta^n})=\limsup_n\frac{\langle x,y_\delta^n\rangle_{\bs}}{d(\bs,x)d(\bs,y_\delta^n)}\\
&=\frac{1}{d(\bs,x)}\limsup_n\langle x,y_\delta^n\rangle_{\bs}\frac{1}{\lim_nd(\bs,y_\delta^n)}\\
&\le \frac{1}{d(\bs,x)}\frac{\int_{B_\delta^c}\langle x,y\rangle_{\bs} d\P(x)}{\P(B_\delta)}\frac{\P(B_\delta)}{\left(\int_{B_\delta}\int_{B_\delta}\langle x,y\rangle_{\bs} d\P\otimes\P(x,y)\right)^{1/2}}.
\end{align*}
Then, since $\int \langle x,y\rangle_{\bs} d\P(y)=0$ by Lemma \ref{lem:sturmNNC}, letting $\delta\rightarrow 0$, one gets
\[
\frac{1}{\P(B_\delta)}\int_{B_\delta^c}\langle x,y\rangle_{\bs} d\P(y)\rightarrow -d^2(\bs,x).
\]
and
\[
\frac{\left(\int_{B_\delta}\int_{B_\delta}\langle x,y\rangle_{\bs} d\P\otimes\P(x,y)\right)^{1/2}}{\P(B_\delta)}\rightarrow d(\bs,x)
\]
Thus, 
\[
\lim_{\delta\rightarrow 0^+}\limsup_n\cos \angle(\uparrow_{\bs}^x,\uparrow_{\bs}^{y_\delta^n})=-1
\]
One can thus choose $(\bar y^n)_n$ a sequence in $(y_\delta^n)_{n,\delta}$ such that $\cos\angle(\uparrow_{\bs}^x,\uparrow_{\bs}^{\bar y^n})\rightarrow -1$.
Since $T_{\bs}\Sp$ is a subspace of an Alexandrov space of curvature bounded below by $0$, we also have 
\begin{align*}
\angle(\uparrow_{\bs}^{\bar y^n},\uparrow_{\bs}^{\bar y^k})&\le 2\pi - \angle(\uparrow_{\bs}^{\bar y^n},\uparrow_{\bs}^{x})-\angle(\uparrow_{\bs}^{x},\uparrow_{\bs}^{\bar y^k})\\
&\rightarrow 0,
\end{align*}
as $n,k\rightarrow \infty$.
Thus $(\bar y^n)_n$ correspond to a Cauchy sequence in the space of direction, and thus admits a limit in $T_{\bs}\Sp$ - since its "norm" also admits a limit $d(\bs,x)$.
Finally, its limit $\bar y$ satisfies $\cos \angle(\uparrow_{\bs}^x,\uparrow_{\bs}^{\bar y})=-1$, and therefore, it is the opposite $\bar y= -x$.
\end{proof}

\begin{lem}[Proposition 1.7 of \cite{sturm_metric_1999} for non separable metric space]
\label{lem:sturmNNC}
Suppose $(\Sp,d)$ is an Alexandrov space with curvature bounded below.
Then, for any probability measure $\Q\in \PS$,
\[
\int \langle x,y\rangle_{\bs} d\Q\otimes \Q(x,y)\ge 0.
\]
Moreover, if $b^\star$ is an exponential barycenter of $\Q$, then for all $x\in\supp \Q$,
\[
\int \langle x,y\rangle_{\bs} d\Q(y)=0.
\]
\end{lem}
\begin{proof}
For brevity, we will adopt the notation $\Q g$ for $\int g d\Q$.

The result for $\Q$ finitely supported is the Lang-Schroeder inequality (Proposition 3.2 in \cite{lang_kirszbrauns_1997}).
Thus, we just need to approximate $\Q\otimes\Q\langle .,. \rangle_{\bs}$ by some $ \Q_n\otimes\Q_n\langle ,.,\rangle_{\bs}  $ for some finitely supported $\Q_n$.

To do this, for i.i.d. random variable $(X_i)_i$ of common law $\Q$, denote $\Q_n$ the empirical measure.
Since $\Sp$ is not separable, we can not apply the fundamental theorem of statistics that ensures almost sure weak convergence of $\Q_n$ to $\Q$.
However, for a measurable function $f:\Sp\times \Sp\rightarrow \R$, such that $\Q\otimes \Q f^2<\infty$, we get the following bound
\begin{align*}
\esp |\Q\otimes \Q f&-\Q_n\otimes \Q_n f|^2\\
&= \frac{1}{n}\sum_{ijkl}\esp [(\Q\otimes\Q f-f(X_i,X_j))(\Q\otimes\Q f-f(X_k,X_l)]\\
&=O(1/n),
\end{align*}
since $n(n-1)(n-2)(n-3)$ of the $n^4$ terms of the sum are equal to 
\[
\esp [(\Q\otimes\Q f-f(X_1,X_2))(\Q\otimes\Q f-f(X_3,X_4))]=0.
\]
And thus, $\Q_n\otimes \Q_n f\rightarrow \Q\otimes \Q f$ in $L^2(\Q^{\otimes\infty})$ and so there exists a (deterministic) probability measure $\Q_n$ supported on $n$ points, such that $\Q_n\otimes \Q_nf\rightarrow \Q\otimes \Q f$.
We thus proved the first result applying $f=\langle .,.\rangle_{\bs}$.

Now applying this first result to the measure $\Q_\e:=\frac{1}{1+\e}\Q + \frac{\e}{1+\e}\delta_x$, we get
\begin{align*}
0&\le (1+\e)\Q_\e\otimes \Q_\e\langle .,.\rangle_{\bs}\\
&=\Q\otimes \Q\langle .,.\rangle_{\bs}+2\e \Q \langle x,.\rangle_{\bs}+\e^2\|x\|_{\bs}^2.
\end{align*}
Letting $\e\rightarrow 0^+$, we get
\[
\Q\langle x,.\rangle_{\bs}\ge 0.
\]
Then equality follows from the hypothesis $\Q\otimes\Q\langle .,.\rangle_{\bs}=0$ meaning that $\bs$ is an exponential barycenter.
\end{proof}

\begin{lem}[Subadditivity, Lemma A.4 of \cite{lang_kirszbrauns_1997}]
\label{lem:subadd}
Let $(\Sp,d)$ be an Alexandrov space with curvature bounded below.
Take $\bs\in\Sp$.
Let $x_1,\dots,x_n\in T'_{\bs}\Sp$ and $U\subset T_{\bs}\Sp$ bounded.
Then, for all $\eps>0$, there exists $y\in T_{\bs}\Sp$ such that for all $u\in U$,
\[
\langle y,u\rangle_{\bs}\le\sum_{i=1}^n\langle x_i,u\rangle_{\bs} +\eps,
\]
and
\[
\|y\|^2\le \sum_{i,j=1}^n\langle x_i,x_j\rangle_{\bs} +\eps.
\]
\end{lem}

\begin{lem}[Approximation]
\label{lem:approx}
Let $U\subset T_{\bs}\Sp$ finite.
Take $B\subset \Sp$ measurable and a probability measure $\P\in\PS$ such that $\P\otimes\P \langle .,.\rangle_{\bs}=0$ and $\P(B)>0$.
Then, there exists a sequence $(y^n)_n$ such that for all $u\in U$
\begin{equation}
\label{eq:apprsc}
\frac{1}{\P(B)}\int_{B^c} \langle u,x\rangle_{\bs} d\P(x) \ge \limsup_n \langle u, y^n\rangle_{\bs}
\end{equation}
and
\begin{equation}
\label{eq:apprnorm}
\frac{1}{\P(B)^2}\int_{B}\int_{B}\langle x,y\rangle_{\bs} d\P\otimes\P(x,y)=\lim_n d^2(\bs,y^n).
\end{equation}
\end{lem}

\begin{proof}
Using the same arguments as in Lemma \ref{lem:sturmNNC}, we see that the empirical measures $(\P_n)_n$ satisfy 
\[
\P_n\otimes\P_n f\rightarrow \P\otimes\P f,
\]
in $L^2(\P^{\otimes\infty})$, for any $f:\Sp\times\Sp\rightarrow\R$.
In particular, taking $f(x,y)=\langle x,y\rangle_{\bs}\mathbf{1}_{B\times B}(x,y)$, the following convergence holds in $L^2(\P^{\otimes\infty})$,
\begin{equation}
\label{eq:PP}
\int_{B}\int_{B}\langle .,.\rangle_{\bs}d\P_n\otimes \P_n \rightarrow \int_{B}\int_{B}\langle .,.\rangle_{\bs} d\P\otimes \P,
\end{equation}
and similarly for $B^c$.
Also, the law of large number ensures that almost surely, for all $u\in U$,
\begin{equation}
\label{eq:Pu}
\int_{B}\langle .,u\rangle_{\bs}d\P_n \rightarrow \int_{B}\langle .,u\rangle_{\bs} d\P,
\end{equation}
and again, the same for $B^c$.
Thus, there exists a subsequence (of a deterministic realisation of) $\P_n$ - that we rename $\P_n$ - such that \eqref{eq:PP} and \eqref{eq:Pu} both hold for all $u\in U$.

Then, applying Lemma \ref{lem:subadd} to finite sum
\[
\frac{1}{\P(B)}\int_{B^c}\langle .,u\rangle_{\bs}d\P_n,
\]
shows that there exists a sequence $(y^n)_n\in T'_{\bs}\Sp$ such that \eqref{eq:apprsc} holds and for a sequence $(\e_n)_n$ s.t. $\e_n\rightarrow 0$,
\[
\|y^n\|_{\bs}^2\le \frac{1}{\P(B)^2}\int_{B^c}\int_{B^c}\langle .,.\rangle_{\bs}d\P_n\otimes\P_n+\eps_n.
\]
Then, applying the same Lemma \ref{lem:subadd} again shows that there exists a sequence $(z^n)_n\subset T'_{\bs}\Sp$, such that
\begin{align}
0&\leftarrow \frac{1}{\P(B)^2}\int\int\langle x,y\rangle_{\bs} d\P_n\otimes\P_n(x,y)\nonumber\\
&=\frac{1}{\P(B)^2}\left(\int_{B}\int_{B}+\int_{B^c}\int_{B^c}+2\int_{B}\int_{B^c}\right)\langle x,y\rangle_{\bs} d\P_n\otimes\P_n(x,y)\nonumber\\
&\ge \|z^n\|_{\bs}^2 + \|y^n\|_{\bs}^2 + 2\langle y^n,z^n\rangle_{\bs}-\eps_n.\label{eq:ineq}
\end{align}
Letting $n\rightarrow \infty$, one obtains
\begin{align*}
0&\ge\lim_n\|z^n\|_{\bs}^2+2\langle y^n,z^n\rangle_{\bs}+\|y^n\|_{\bs}^2\\
&\ge\lim_n \|z^n\|_{\bs}^2-2\|y^n\|_{\bs}\|z^n\|_{\bs}+\|y^n\|_{\bs}^2\\
&=\lim_n(\|z^n\|_{\bs}-\|y^n\|_{\bs})^2\ge 0.
\end{align*}
and which shows $\lim_n \|y^n\|=\lim_n\|z^n\|$ and also that \eqref{eq:ineq} becomes an equality at the limit and therefore
\[
\lim_n\|z^n\|_{\bs}^2=\frac{1}{\P(B)^2}\int_{B}\int_{B}\langle x,y\rangle_{\bs}d\P\otimes\P(x,y)
\]
\end{proof}

\printbibliography[title=References]
\end{refsection}

\end{document}